\documentclass[12pt]{amsart}
\usepackage{amssymb}
\usepackage{fullpage}
\usepackage{enumitem}
\usepackage{bm}
\usepackage[all]{xy}
\usepackage{rotating}
\usepackage{graphicx}
\usepackage{stackrel}

\def \eff {{\text{eff}}}
\def \short {{\text{short}}}
\def \tall  {{\text{tall}}}
\DeclareMathOperator \image {image}

\def \t    {{\mathfrak t}}
\def \ft    {{\mathfrak t}}
\def \h    {{\mathfrak h}}
\def \mubar    {\ol{\mu}}

\def \ol   {\overline}
\def \del  {\partial} 
\def \ssminus {\smallsetminus}

\def \Z    {{\mathbb Z}}

\def \R    {{\mathbb R}}
\def \C    {{\mathbb C}}
\def \P    {{\mathbb P}}
\def \CP    {{\mathbb C}{\mathbb P}}

\def \actson {\  \rotatebox[origin=c]{-90}{$\circlearrowright$}\  }

\DeclareMathOperator \interior {interior}

\DeclareMathOperator \SU {SU}
\DeclareMathOperator \SO {SO}

\newif\ifdebug                                                      %
\debugfalse
\newcommand{\printname}[1] {\ifdebug 
         \smash{\makebox[0pt]{\hspace{-1.0in}\raisebox{8pt}{\tiny #1}}} \fi}

\newcommand{\labell}[1] {\label{#1}\printname{#1}}
\numberwithin{equation}{section}
\newtheorem {Theorem}[equation]                   {Theorem}
\newtheorem*{Theorem*}                   {Theorem}

\newtheorem {Lemma}[equation]           {Lemma}

\newtheorem {Corollary} [equation]      {Corollary}
\newtheorem* {Corollary*}               {Corollary}

\newtheorem {Proposition} [equation]    {Proposition}

\newtheorem* {Lemma*}                    {Lemma}

\newtheorem* {Assumption*} {Assumption}
\newtheorem* {GlobalAssumptions*} {Global Assumptions}

\theoremstyle{definition}

\theoremstyle{remark}
\newtheorem{Remark}[equation]{Remark}
\newtheorem*{Remark*}{Remark}

\begin{document}

\title{Topology of complexity one quotients}

\author{Yael Karshon}
\address{Department of Mathematics, University of Toronto, 
40 St.\ George Street, 6th floor,
Toronto Ontario M5S 2E4, Canada}
\email{karshon@math.toronto.edu}

\author[Susan Tolman]{Susan Tolman}
\address{Department of Mathematics, 
University of Illinois at Urbana-Champaign, Urbana, IL 61801}
\email{stolman@math.uiuc.edu}

\begin{abstract}
We describe of the topology of the geometric quotients
of $2n$ dimensional compact connected symplectic manifolds
with $n-1$ dimensional torus actions.
When the isotropy weights at each fixed point are in general position,
the quotient is homeomorphic to a sphere.
\end{abstract}

\maketitle

%-------------------------------------------------------------------------
\section{Introduction}

This paper is a byproduct of our work on the classification 
of complexity one Hamiltonian torus actions
\cite{karshon:periodic,KT:locun,KT:globun,KT:globex,KT:short},
but, in fact, it relies only on elementary aspects of such actions.
It is motivated by a number of recent works by toric topologists
(specifically, the papers
\cite{Buchstaber-Terzic:Gr-CP5,Buchstaber-Terzic:2n-k,
Buchstaber-Terzic:Gr,Ayzenberg:local} by Buchstaber and Terzic
and by Ayzenberg)
that explore the topology of the geometric quotients 
of manifolds with certain torus actions. 
Our purpose in this paper is to highlight topological aspects 
of related works in equivariant symplectic geometry
and to illustrate how equivariant symplectic methods 
reproduce some of the recent results in toric topology
and yield new examples.

Similar results were recently obtained by Hendrik S\"uss \cite{Suss}
from the point of view of algebraic geometry.

The examples  Buchstaber and Terzic studied include the quotient of the 
Grassmannian 
of complex 2-planes in $\C^4$ by its standard torus action, which
they showed is homeomorphic to a five dimensional sphere,
and the quotient of the manifold of complete flags in $\C^3$
by its standard torus action, which they showed is  homeomorphic to a four dimensional sphere.
We exhibit these examples as special cases of a more general phenomenon:
for any Hamiltonian action of a torus $T$ 
on a compact symplectic manifold $M$, 
if the reduced spaces over the interior of the momentum polytope 
are two dimensional and those over the boundary are single points---this
condition holds if and only if the dimension of the torus
is one less than the dimension of the manifold
and at each fixed point the isotropy weights are in general position---then
the geometric quotient $M/T$ is homeomorphic to a sphere.

\subsection*{Acknowledgement}
This work
is partially funded by the Natural Sciences and Engineering 
Research Council of Canada.
We are grateful to Svjetlana Terzic, Nikita Klemyatin,
and Anton Ayzenberg for helpful discussions.
We are grateful to Hendrik S\"uss for useful comments on our draft.
We wish Victor Buchstaber a happy birthday.

%-------------------------------------------------------------------------
\section{Background and main result}
\labell{sec:background}

Let $T$ be a torus and $\t^*$ the dual to its Lie algebra.

Let $(M,\omega)$ be a symplectic manifold with a $T$ action
and with a momentum map ${\mu \colon M \to \t^*}$.
Such an action is called~\textbf{Hamiltonian}.
We recall the definitions and properties
of Hamiltonian torus actions in Appendix~\ref{app:Hamiltonian actions}.
In particular, 
the momentum map $\mu$ is constant on $T$ orbits,
so it induces a map,
which is sometimes called the \textbf{orbital momentum map},
on the geometric quotient,
$$ \mubar \colon M/T \to \t^*.$$
In this paper we always assume that $M$ is compact\footnote
{
Many of the results in this paper remain true
when $M$ is not necessarily compact
but $\mu$ is proper as a map to some convex subset of $\t^*$.
}
and connected.
Since $M$ is compact, the fixed set $M^T$ is not empty.
To see this, fix a  vector $\xi \in \ft$ that generates a dense one-parameter 
subgroup.
Any point $p \in M$ on which the function
$\left < \mu(\cdot), \xi \right> \colon M \to \R$ 
achieves its minimal value is a fixed point for the one-parameter subgroup,
and hence for $T$.

%-----
\medskip\noindent\textbf{Local normal form and the convexity package.}

The local structure of a Hamiltonian torus action is governed
by the local normal form, which describes a neighbourhood of an orbit
up to an equivariant symplectomorphism that preserves momentum maps.
We recall the statement of the local normal form in an appendix.

We denote 
$$ \mathbf{\Delta} := \image{ \mu}. $$
We will need the following theorem and corollary.

\begin{Theorem}[\textbf{Convexity package}]
$\Delta$ is a rational\footnote{
``Rational'' means that the facets have rational conormal vectors.}
convex polytope, 
and the map ${\mu \colon M \to \Delta}$ is open and has connected fibres.
\end{Theorem}

\begin{Corollary} \labell{corollary of convexity}
For any convex subset $C$ of $\t^*$, 
the preimage $\mu^{-1}(C)$ is connected.
\end{Corollary}

The local normal form is due to Guillemin-Sternberg~\cite{GS1984}
and Marle~\cite{marle1985}.

The convexity package is due to Guillemin-Sternberg and Atiyah.
Relevant references include the papers~\cite{
GS:convexity, Atiyah:convexity, CDM, lerman-tolman, HNP, LMTW,
Bjorndahl-Karshon, BOR1, BOR2}.
The corollary follows from the 
(convexity of $C$ and $\Delta$, hence) connectedness of $C \cap \Delta$
by the following exercise in point set topology:
Given an continuous open map with connected fibres, 
the preimage of any connected subset of the image is connected.

%-----
\medskip\noindent\textbf{Principal orbit types over faces 
and in level sets; the complexity}

We continue to assume that $M$ is compact and connected.
Let $T_\eff$ be the quotient of $T$ by the kernel of the action.
Because $M$ is connected, it has a connected open dense subset
where the action of $T_\eff$ is free.
The formula for the momentum map implies that
the affine span of the momentum image of $M$ 
is a translation of the annihilator in $\t^*$
of the Lie algebra of the kernel of the action.  In particular,
\begin{equation} \labell{dim F}
 \dim T_\eff = \dim \Delta .
\end{equation}

The action is \textbf{toric} if $\dim T_\eff = \frac{1}{2} \dim M$.
More generally, 
the \textbf{complexity} of the action is 
$\frac{1}{2} \dim M - \dim T_\eff$;
it measures how far the action is from being toric.

\begin{Lemma} \labell{preimage of face}
For every face\footnote{
Because the convex set $\Delta$ is locally polyhedral,
a subset $F$ of $\Delta$ is a face
if and only if it is equal either to $\Delta$
or to the intersection of $\Delta$ with a supporting hyperplane
(a hyperplane that meets $\Delta$ and such that 
one of the two closed half-spaces that it bounds contains $\Delta$).
}
$F$ of $\Delta$, 
its preimage $M_F$ in $M$, with the structures induced from $M$,
is a compact
connected symplectic manifold with a Hamiltonian $T$ action.
\end{Lemma}

\begin{proof}

By the definition of ``face'',
there exist $\xi \in \ft$ and $a \in \R$
such that $\langle \mu(p) , \xi \rangle \geq a$ for all $p \in M$, with
equality exactly if $p \in M_F$.

Given any $p \in M_F$, let $\h$ be the Lie algebra of its stabilizer,
and let $\eta_j \in \h^*$ be the isotropy weights at $p$
(see Appendix~\ref{app:lnf}).
By the local normal form theorem,
the fact that
$\langle \mu(q) , \xi \rangle \geq \langle \mu(p) , \xi \rangle$
for all $q$ near $p$
implies that $\xi \in \h$
and that
$\langle \eta_j, \xi \rangle \geq 0$ for all $j$.
The local normal form theorem then implies
that the intersection of $M_F$
with a neighbourhood of the orbit of $p$ 
is a $T$ invariant symplectic submanifold.

By Corollary~\ref{corollary of convexity}, $M_F$ is connected.

\end{proof}

\begin{Remark}\labell{preimage of face 2}
Let $K$ be the identity component of the kernel of the $T$ action on $M_F$.
By the definition of the momentum map,
the affine span of $F$ is a translation of the annihilator in $\ft^*$
of the Lie algebra of $K$. Moreover,
$M_F$ is a connected component of  $M^K$, the set of points fixed by $K$,
because the component of $M^K$ containing $M_F$ must lie in the preimage
of the affine span of $F$.
In particular, the preimage in $M$ of any vertex of $\Delta$
is a component of the fixed point set $M^T$.
\end{Remark}

\begin{Lemma} \labell{criterion for complexity}
Given any face $F$ and any fixed point $p$ in the preimage $M_F$,
the complexity of the $T$ action on $M_F$  is 
the number of isotropy weights at $p$ that are parallel to $F$ 
minus the dimension of $F$.
Moreover, the linear span of the weights that are parallel to $F$
is a translation of the affine span of $F$.
\end{Lemma}

\begin{proof}
By Lemma \ref{preimage of face}, the preimage $M_F$ of $F$ in $M$
is a compact connected symplectic manifold with a Hamiltonian $T$ action.
Let $K$ be the identity component of the kernel of the $T$ action on $M_F$.
By Remark~\ref{preimage of face 2}, the affine span of $F$
is a translation of the annihilator in $\ft^*$ of the Lie algebra of $K$,
and $M_F$ is a connected component of $M^K$.
Hence, the dimension of $T/K$ is the dimension of $F$,
and the weights for the action on $T_p M_F$
are those weights for the action on $T_p M$ that 
annihilate the Lie algebra of $K$, or equivalently, are parallel to $F$.
Therefore, the dimension of $M_F$ is twice the number of such weights.

Finally, by the local normal form theorem,
there is a neighbourhood of $p$ in $M_F$
that is equivariantly symplectomorphic 
to $T_pM_F$.
Since $K$ is the identity component of the stabilizer
of an open dense set of points in $M_F$,
the identity component of the kernel of the isotropy representation on $T_pM_F$
is also $K$.
Hence, the isotropy weights at $p$
span the annihilator in $\t^*$ of the Lie algebra of $K$.
\end{proof}

\begin{Corollary} \labell{complexity of face}
Let $M_F$ and $M_{F'}$ be the preimage of faces $F$ and $F'$ of $\Delta$,
respectively.
If $F \subseteq F'$, then the complexity of $M_F$ is less than or equal to
the complexity of $M_{F'}$.
\end{Corollary}

\begin{proof}
By Lemma~\ref{preimage of face},
$M_F$ and $M_{F'}$ are compact connected symplectic manifolds
with Hamiltonian $T$ actions. 
Consider a fixed point $p \in M_F$.
Since the linear span of the isotropy weights at $p$
that are parallel to $F'$
is a translation of the affine span of $F'$,
the number of weights that are parallel to $F'$ but not $F$ must
be greater than or equal to the codimension of $F$ in~$F'$.
\end{proof}

Given a point $\beta \in \t^*$, 
let $M_\beta := \mubar^{-1}(\{\beta\}) = \mu^{-1}(\{\beta\})/T$
be the \textbf{reduced space} at $\beta$.
If $T_\eff$ acts freely on $\mu^{-1}(\{\beta\})$, 
then $M_\beta$ is naturally a manifold.
More generally, the following holds.

\begin{Lemma} \labell{reduced over interior}\ 
Given a point $\beta$ in the relative interior of $\Delta$,
the set of free orbits in the reduced space $M_\beta$
is a connected open dense subset of $M_\beta$;
moreover, it is naturally\footnote
{Explicitly, there exists a unique manifold structure
on the set of free orbits in $M_\beta$ such that a real valued function
on this set is smooth if and only if its pullback 
to the preimage in $\mu^{-1}(\{\beta\})$ extends to a smooth function
on an open subset of $M$.
} 
a $2k$ dimensional manifold,
where $k$ is the complexity of the $T$ action on $M$.
\end{Lemma}

\begin{proof}
This consequence of the local normal form theorem
and the convexity package
is proved by Lerman and Sjamaar in~\cite{lerman-sjamaar}.
\end{proof}

The \textbf{dimension} of a reduced space $M_\beta$
is  the dimension of an open dense subset of $M_\beta$ that is a manifold;
it is well defined,
by Lemmas~\ref{preimage of face} and~\ref{reduced over interior}.
For any nonnegative integer $k$,
denote by $\mathbf{\Delta}_{\mathbf{k}}$ the set of points $\beta$ in $\Delta$
such that $\dim M_\beta = 2k$,
and denote $\mathbf{\Delta}_{\mathbf{\leq k}} := 
\Delta_{0} \cup \ldots \cup \Delta_{k}$.
By the connectedness of the momentum map fibres,
$\Delta_{0}$ 
is the set of points $\beta$ in $\Delta$
such that the reduced space $M_\beta$ consists of a single orbit.

%-----

%\newpage

\begin{Lemma} \labell{union of faces}
For any nonnegative integer $k$,
the set $\Delta_{\leq k}$ is a union of faces of $\Delta$.
Consequently, there exists an open convex subset $U$ of $\t^*$
such that $\Delta \ssminus \Delta_{\leq k} = \Delta \cap U$.
\end{Lemma}

\begin{proof}
By Lemma~\ref{preimage of face}, 
the preimage $M_F := \mu^{-1}(F)$ of each face $F$ of $\Delta$
is a compact connected symplectic manifold with a Hamiltonian $T$ action.
Hence, by Lemma~\ref{reduced over interior},
each $\Delta_k$ is the union of the relative interiors of those faces $F$
for which the complexity of $M_F$ is equal to $k$.
The first claim then follows from Corollary~\ref{complexity of face}.

To prove the second claim, for each face $F$ in $\Delta_{\leq k}$
choose a supporting hyperplane $H_F$ of $\Delta$ 
such that $F = H_F \cap \Delta$.
Then the intersection $U$ of the appropriate open half-spaces bound by
these hyperplanes is an open convex set.
\end{proof}

\begin{Remark}[Toric manifolds] \labell{toric}
If we assume that the  $T$ action on $M$ is toric, then the quotient $M/T$ is
homeomorphic to the disk $D^n$, where $n = \frac{1}{2} \dim M$.
To see this, first note that Lemma~\ref{preimage of face} and
Corollary~\ref{complexity of face}
 together show that 
the preimage $M_F := \mu^{-1}(F)$ of each face $F$ of $\Delta$ is a symplectic
toric manifold.  
Hence, by Lemma~\ref{reduced over interior},  the reduced space $M_\beta$ is a
point for all $\beta \in \Delta$,
that is, $\Delta_0 = \Delta$.  Thus, the orbital momentum map $\overline{\mu}
\colon M/T \to \Delta$ is
a bijection; since it is proper and continuous, this implies that it is a
homeomorphism.  Since $\Delta$ is a convex polytope,
this proves the claim.

More generally, 
consider a complete unimodular fan in $\R^n$.
Even if the fan does not correspond to any convex polytope,
we can construct a  complex toric manifold $M$ from the
fan, as described by Audin in \cite{audin}.
The geometric quotient $M/T$ is still homeomorphic to a sphere;
see \cite[Lemma~3.2]{KT:toric}.
\end{Remark}

{
A collection of vectors in the vector space $\t^*$
is \textbf{in general position}
if every sub-collection of size $< \dim \t^*$ is linearly independent. 

\begin{Lemma} \labell{toric strata}
Assume that $M$ is compact.
\begin{enumerate}[leftmargin=1cm]
\item
Assume that there exists an isolated fixed point in $M$
whose momentum image is a vertex of $\Delta$;
in particular, this holds if the fixed points in $M$ are isolated.
Then $\Delta_0 \neq \emptyset$.

\item
Assume that the $T$ action on $M$ has complexity $\geq 1$
and that the isotropy weights at every fixed point 
are in general position. Then $\Delta_0 = \del \Delta$.

\end{enumerate}
\end{Lemma}

\begin{proof}
Part~(1) is a consequence of the following two facts.
First, since $M$ is compact, its momentum image $\Delta$ has a vertex.
Second, by Remark~\ref{preimage of face 2},
the preimage of any vertex of $\Delta$
is a connected component of the fixed point set $M^T$.

We now prove Part (2).
First, consider $\beta \in \del \Delta$.
Let $F \subsetneq \Delta$ be the face whose relative interior contains $\beta$.
By Lemma~\ref{preimage of face}, the preimage $M_F$ of $F$ in $M$ 
is a compact connected symplectic $T$ manifold with a Hamiltonian $T$ action.
So it has a fixed point $p$.
Since $\dim F < \dim \t^*$ and the isotropy weights at $p$ 
are in general position,
Lemma~\ref{criterion for complexity} implies that $M_F$ is toric.  
Therefore, by Lemma~\ref{reduced over interior}, 
$\beta \in \Delta_0$.
In contrast, if $\beta$ is in the relative interior of $\Delta$
then, since the action of $T$ on $M$ is not toric, 
Lemma~\ref{reduced over interior} implies that 
$\beta$ is not in $\Delta_0$. 
\end{proof}

\begin{Remark} \labell{general position complexity 1}
In Part~(2) of Lemma~\ref{toric strata}, 
if the complexity of the $T$ action on $M$ is \emph{equal} to one,
then the converse is true too, so $\Delta_0 = \del \Delta$
if and only if the isotropy weights at every fixed point 
are in general position.
\end{Remark}

}

%-----

When the complexity of the Hamiltonian $T$ action is equal to one, 
we denote by $\mathbf{\Delta}_\textbf{\short}$ the set of points 
in $\Delta$
whose reduced space contains a single orbit
and by $\mathbf{\Delta}_\textbf{\tall}$ the set of points in $\Delta$
whose reduced space is two dimensional.
Thus, 
$$
 \Delta_\short = \Delta_{0} \quad \text{ and } \quad
 \Delta = \Delta_\short \sqcup \Delta_\tall.$$
By Lemma~\ref{union of faces}, $\Delta_\short$ is closed,

\begin{Proposition} \labell{prop:tall}
Let $T$ be a torus and $\t^*$ the dual to its Lie algebra.
Let $M$ be a compact connected symplectic manifold with a $T$ action
and with a momentum map $\mu \colon M \to \t^*$ with image $\Delta$.
Assume that the action has complexity one.

Then
there exists a connected closed oriented surface $\Sigma$ and a homeomorphism
$$ (M/T)_\tall \to \Delta_\tall \times \Sigma$$
that intertwines the orbital momentum map $\mubar$ 
with the projection map to $\Delta_\tall$.

If $\Delta_\short$ is non-empty, then $\Sigma$ is a two-sphere.
\end{Proposition}

\begin{proof}
By Lemma~\ref{union of faces},
there exists a convex open subset $U$ of $\t^*$
such that $\Delta_\tall = \Delta \cap U$.
The first part of Proposition~2.2 of~\cite{KT:globun} then implies 
that there is a homeomorphism $(M/T)_\tall \to \Delta_\tall \times \Sigma$
as required.
By \cite[Lemma~5.7]{KT:locun}, 
if $\Delta_\short$ is non-empty, then $\Sigma$ is a sphere.
\end{proof}

We now state our main theorem. 

\begin{Theorem} \labell{thm:general}
Let $T$ be a torus and $\t^*$ the dual to its Lie algebra.
Let $M$ be a $2n$ dimensional 
compact connected symplectic manifold with a $T$ action
and with a momentum map $\mu \colon M \to \t^*$ with image $\Delta$.
Assume that the action has complexity one.
Then
there exist a connected closed oriented surface $\Sigma$ and a homeomorphism
$$ M/T \to (\Delta \times \Sigma) / \!\!\sim,$$
where $\sim$ is the finest equivalence relation
with $(x,y) \sim (x,y')$ if $x \in \Delta_{\short}$.
Moreover,
\begin{itemize} 
\item[(i)]
If $\Delta_\short$ is non-empty, then $\Sigma$ is a two-sphere.
\item[(ii)]
If $\Delta_\short = \del \Delta$,
then $M/T$ is homeomorphic to the $(n+1)$-sphere.
\end{itemize}
\end{Theorem}

{
\begin{proof}
By Proposition \ref{prop:tall},
there exists a connected closed oriented surface $\Sigma$ and a homeomorphism
$$ (M/T)_\tall \xrightarrow{} \Delta_\tall \times \Sigma $$
that intertwines the orbital momentum map $\mubar$
and the projection map to $\Delta_\tall$.
Since $\Delta = \Delta_\short \sqcup \Delta_\tall$
and $\Delta_\short$ consists of those $\beta$ such that $M_\beta$
consists of a single orbit,
this homeomorphism extends to a unique bijection 
$$ f \colon M/T \to (\Delta \times \Sigma) / \!\!\sim$$
that intertwines the orbital momentum map $\mubar$
with the map 
$\pi \colon (\Delta \times \Sigma)/\!\!\sim \ \xrightarrow{} \t^*$
induced by the projection to $\Delta$.
Since $(M/T)_\tall$ is open in $M/T$ and $\Delta_\tall$ is open in $\Delta$,
the map $f$ is continuous and open at every point of $(M/T)_\tall$.

Since $M$ and $\Sigma$ are compact,
the maps $\mubar \colon M/T \to \t^*$ 
and $\pi \colon (\Delta \times \Sigma)/\!\!\sim \, \xrightarrow{} \t^*$
are proper.
Since $\t^*$ is a locally compact Hausdorff space,
the proper maps $\mubar$ and $\pi$ to $\t^*$ are closed.
Since $\pi$ is closed, $f$ is continuous at every point of $(M/T)_\short$.
Since $\mubar$ is closed and $f$ is onto, 
$f$ is open at every point of $(M/T)_\short$.

% \bigskip

Part (i) follows from the last claim of Proposition~\ref{prop:tall}.

% \bigskip

We now prove Part (ii). 
Since $M$ is compact and connected, $\Delta$ is a convex polytope; 
hence, it is homeomorphic to $D^{n-1}$, where $\dim M = 2n$.
Therefore, the map from $D^{n-1} \times S^2$ that sends $(x,z)$ 
to $(x, \sqrt{1 - |x|^2}\, z)$ induces a continuous proper map
from $(\Delta \times \Sigma)/\!\!\sim$ to $S^{n+1}$.
If $\Delta_\short = \del \Delta$, this map is a bijection.
Since $S^{n+1}$ is a locally compact Hausdorff space, 
being a continuous proper bijection 
implies that this map is a homeomorphism.
\end{proof}
}

\begin{Remark} \labell{join}
Part (ii) of Theorem~\ref{thm:general} 
can be rephrased as follows: If $\Delta_\short = \del \Delta$, then
$M/T$ is homeomorphic to the join $\del \Delta * S^2$. To see this, recall
that the join $A * B$ of two
topological spaces $A$ and $B$ is the quotient of $A \times B \times [0,1]$
under the identifications 
$(a,b,0) \sim (a',b,0)$ and $(a,b,1) \sim (a,b',1)$ for all $a,a' \in A$ and
$b,b' \in B$.  We may assume without loss of generality that $0 \in
\interior \Delta$.  Then, since $\Delta$ is convex,
the map $\del \Delta \times B \times [0,1] \to \Delta \times B$ that is
defined by $(a,b,t) \mapsto (ta,b)$
descends to a continuous proper bijection $\del \Delta \times B \to (\Delta
\times B)/\sim$,
where here $\sim$ is the finest equivalence relation with $(x,y) \sim (x,y')$
if $x \in \del \Delta$.
When $B$ is a locally compact Hausdorff space, 
this bijection is a homeomorphism.
\end{Remark}

\begin{Corollary} \labell{interjections}
Let $T$ be a torus and $\t^*$ the dual to its Lie algebra.
Let $M$ be a compact connected symplectic manifold
with a $T$ action and a momentum map $\mu \colon M \to \t^*$
with image~$\Delta$.  Assume that the action has complexity one.
\begin{enumerate}
\item[(a)]
Assume that there exists an isolated fixed point in $M$
whose momentum image is a vertex of~$\Delta$;
in particular, this holds if the fixed points in $M$ are isolated. 
Then $M/T$ is homeomorphic to $(\Delta \times S^2)/\!\!\sim$,
where $\sim$ is the finest equivalence relation
with $(x,y) \sim (x,y')$ if $x \in \Delta_\short$.

\item[(b)]
Assume that the isotropy weights at every fixed point are in general position.
Then $M/T$ is homeomorphic to a sphere.

\end{enumerate}
\end{Corollary}

\begin{proof}
Part~(a) follows from 
Part~(1) of Lemma~\ref{toric strata}
and Part (i) of Theorem \ref{thm:general}.
Part~(b) follows from 
Part~(2) of Lemma~\ref{toric strata}
and Part (ii) of Theorem \ref{thm:general}.
\end{proof}

In Part (b) of Corollary~\ref{interjections},
the fact that $M/T$ is a topological manifold
already follows from a result of Ayzenberg~\cite{Ayzenberg:local}.
Ayzenberg's work also implies that
if the action extends to a toric action
then $M/T$ is homeomorphic to a sphere.

\begin{Corollary} \labell{M4}
Let the circle $S^1$ act on a compact connected symplectic four-manifold
$(M,\omega)$ with momentum map $\mu \colon M \to \R$.  
Then exactly one of the following is true.
\begin{enumerate}
\item The fixed point set is finite 
and $M/T$ is homeomorphic to a three-sphere.
\item The fixed point set contains one surface, which is a sphere, 
and $M/T$ is homeomorphic to a three-disk.
\item The fixed point set contains two surfaces
that have the same genus $g$,
and $M/T$ is homeomorphic to $[0,1] \times \Sigma$
where $\Sigma$ is a surface of genus $g$.
\end{enumerate}
\end{Corollary}

\begin{proof}
By rescaling $\omega$ if necessary, we may assume that the momentum image is
the interval $[0,1]$.
Since $0$ and $1$ are vertices of $[0,1]$, 
Lemma~\ref{preimage of face} and Remark~\ref{preimage of face 2} 
imply that each of $\mu^{-1}(\{0\})$ and $\mu^{-1}(\{1\})$ 
is a connected component of the fixed point set
that is either a single point or a fixed surface.
By the local normal form theorem, a fixed point that is not isolated
is a local minimum or local maximum of the momentum map;
since by the convexity package
the momentum map is open as a map to its image $[0,1]$,
such a fixed point must be mapped to $0$ or to $1$.
Hence, there are at most two components of the fixed point set
that are not isolated fixed points, 
and each of them is mapped to $0$ or to $1$.

Assume first that the fixed point set contains no surfaces.
Then the fixed points are isolated, 
and so none of the isotropy weights at any fixed point are zero.
Hence, $M/T$ is homeomorphic to a three-sphere 
by Part (b) of Corollary~\ref{interjections}.

Assume now that the fixed point set contains exactly one surface $\Sigma$.
By replacing $\omega$ by $-\omega$ if necessary, we may assume that
$\mu(\Sigma) = 1$.
Since $\Sigma$ is the only fixed surface, $\mu^{-1}(\{0\})$ is an isolated
fixed point.
Hence, $\Delta_\short = \{0\}$.
By Part (a) of Corollary~\ref{interjections}, this implies the $M/S^1$ is
homeomorphic to
$[0,1] \times S^2/\!\!\sim$, where $\sim$ is the finest equivalence relation
such that $(0,x) \sim (0,x')$.
Define a map from $[0,1] \times S^2$ to $\R^3$ by $(t,x) \mapsto tx$, where
we identify $S^2$ with
the unit sphere in $\R^3$. This induces a homeomorphism 
from $[0,1] \times S^2/\!\!\sim$ to the three-disk $D^3$,
and hence from $M/T$ to $D^3$.

Finally, assume that the fixed point set contains two surfaces, 
$\Sigma$ and $\Sigma'$.
By the first paragraph, we may assume that $\mu^{-1}(\{0\}) = \Sigma$ and
$\mu^{-1}(\{1\}) = \Sigma'$.
Hence, $\Delta_\short$ is empty,
and so Theorem \ref{thm:general} implies that $M/T$ 
is homeomorphic to $[0,1] \times \Sigma_g$ 
for some oriented surface~$\Sigma_g$.
\end{proof}

%-------------------------------------------------------------------------
\section{Examples}
\labell{sec:examples}

\medskip\noindent\textbf{Examples}

\medskip

{
In Table~\ref{examples} we list some examples of symplectic torus actions 
and their geometric quotients.

\renewcommand{\arraystretch}{1.6} 
\begin{table}%[h]
$$
\begin{array}{c|c|c|c|c|c|}
   & M & \dim_\R M & T & \text{complexity} & M/T \\ \hline \hline
 (1) & G_{2}(\C^4) = \{ E^2_\C \subset \C^4 \} & 8 & (S^1)^4/\text{diag} & 1
   & \stackrel{\text{homeo}}{\cong} S^5 \\ \hline
 (2) & 
F_3 = \{ L^1_\C \subset E^2_\C \subset \C^3 \} & 6 & (S^1)^3/\text{diag} & 1
   & \stackrel{\text{homeo}}{\cong} S^4 \\ \hline
 (3) & 
   G_2^+(\R^5) = \{ 
            E^2_{\text{oriented}} \subset \R^2 \times \R^2 \times \R \}
   & 6 & (S^1)^2 & 1
   & \stackrel{\text{homeo}}{\cong} S^4 \\ \hline
 (4) & 
   \includegraphics[scale=.3,trim=0 0 0 -6mm]{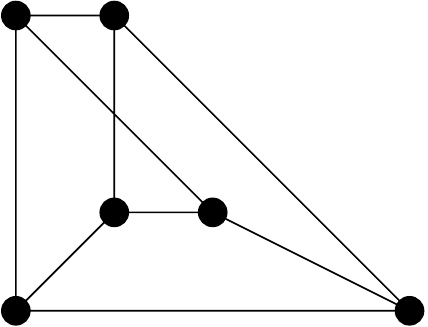}
   & 6 & (S^1)^2 & 1
   & \stackrel{\text{homeo}}{\cong} S^4 \\ \hline
 (5) & 
\begin{minipage}{2.5in}
\vspace{.2\baselineskip}
   $  S^1 \actson (S^2)^2 $ \\ 
\vspace{-.5\baselineskip}
   $ a \cdot (u,v) = (a \cdot u , a \cdot v) $
\vspace{.7\baselineskip}
\end{minipage}
   & 4 & S^1  & 1
   & \stackrel{\text{homeo}}{\cong} S^3  \\ \hline
 (6) & 
\begin{minipage}{2.5in}
\vspace{.2\baselineskip}
   $S^1 \actson \CP^2$ \\
\vspace{-.2\baselineskip}
   $a\cdot [z_0:z_1:z_2] = [az_0:z_1:z_2]$
\vspace{.4\baselineskip}
\end{minipage}
   & 4 & S^1 & 1
   & \stackrel{\text{homeo}}{\cong} D^3 \\ \hline
 (7) & 
   S^1 \actson S^2 \times \Sigma_g & 4 & S^1 & 1
   & \stackrel{\text{homeo}}{\cong} I \times \Sigma_g \\ \hline
 (8) & 
\begin{minipage}{2.5in}
\vspace{.2\baselineskip}
   $ S^1 \times S^1 \actson (S^2)^3 $ \\ 
\vspace{-.5\baselineskip}
   $ (a,b) \cdot (u,v,w) = (a \cdot u , a \cdot v, b \cdot w) $
\vspace{.7\baselineskip}
\end{minipage}
   & 6 & S^1 \times S^1 & 1
   & \stackrel{\text{homeo}}{\cong} S^3 \times I \\ \hline
%

%
% trim = left bottom right top
%
 (9) & \includegraphics[scale=.3,trim=0 0 0 -4mm]{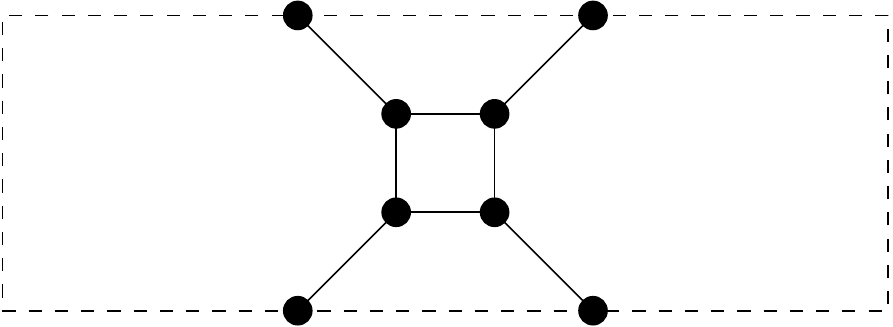}
   & 6 & (S^1)^2 & 1 
   & \stackrel{\text{homeo}}{\cong} I \times I \times \Sigma_g \\ \hline
\hline

 (10) & \CP^5 = \P(\bigwedge^2\C^4)
    & 10 & (S^1)^4/\text{diag} & 2
    & \stackrel[\text{\cite{Buchstaber-Terzic:Gr-CP5}}]
               {\text{homeo}}{\cong} S^2 * \CP^2 \\ \hline

 (11) & 
   G_{2}(\C^5) = \{ E^2_\C \subset \C^5 \} & 12 & (S^1)^5/\text{diag} & 2
   & {\text{\cite{Buchstaber-Terzic:Gr,Suss2}}}
\\ \hline

\end{array}
$$
\caption{Examples of geometric quotients}
\labell{examples}
\end{table}
}

\medskip

We now discuss these examples and give some references.
\begin{itemize}[leftmargin=1cm,itemsep=8pt]

%-------------------------------------------------------------
\item[(1)]
Let $M$ be the Grassmannian of complex 2-planes in $\C^4$,
with the three dimensional torus action induced from the 
standard action of $(S^1)^4$ on $\C^4$.
Then $M/T$ is homeomorphic to the sphere $S^5$;
this is shown in \cite{Buchstaber-Terzic:Gr-CP5}
and revisited in \cite[Section~10]{Buchstaber-Terzic:2n-k}.
Alternatively, we can identify $M$ equivariantly with
a coadjoint orbit in $\SU(4)$,
where $T$ is a maximal torus acting through the coadjoint action.
There is a natural symplectic structure on every coadjoint orbit
of any Lie group, and the coadjoint action is Hamiltonian.
Hence, since the isotropy weights at each fixed point are
in general position,
we can apply Corollary~\ref{interjections}
and conclude that $M/T$ is homeomorphic to the sphere $S^5$.

%-------------------------------------------------------------
\item[(2)]
Let $M$  be the manifold of complete complex flags in $\C^3$, 
with the two dimensional torus action that is induced from 
standard action of $(S^1)^3$ on $\C^3$.
Then $M/T$ is homeomorphic to the sphere $S^4$;
this is shown in \cite{Buchstaber-Terzic:2n-k}.
Alternatively, $M$ is a coadjoint orbit of $\SU(3)$, 
and so---as in the previous example---$M/T$ is homeomorphic to
the sphere $S^4$ 
by Corollary~\ref{interjections}.

%-------------------------------------------------------------
\item[(3)]
Let $M$  be the Grassmannian of oriented (real) 2-planes 
in $\R^5 \cong (\R^2)^2 \times \R$, 
with the two dimensional torus action
that is induced from 
the standard action of $(S^1)^2$ on  the $(\R^2)^2$ factor.
By identifying $M$ with a coadjoint orbit of $\SO(5)$,
we obtain a symplectic form such that the action is Hamiltonian.
Since the isotropy weights at each fixed point are in
general position, $M/T$ is homeomorphic to the sphere $S^4$ 
by Corollary~\ref{interjections}.
For more details, see, e.g., Section~14 of~\cite{KT:locun}.

%-------------------------------------------------------------
\item[(4)]
Let $(M,\omega,\mu)$ be the compact symplectic six manifold
with Hamiltonian $(S^1)^2$ action constructed in \cite{tolman:example}.
The picture drawn in the table 
shows the momentum map images of the orbit type strata.
The solid dots are the images of isolated fixed points,
and the segments are the images 
of 2-spheres with circle stabilizer.
As the second author showed in  \cite{tolman:example},
$M$ does not admit any invariant K\"ahler structure.
As in the previous examples, $M/T$ is homeomorphic to the sphere $S^4$ 
by Corollary~\ref{interjections}.

%-------------------------------------------------------------
\item[(5)]
Let $M$ be the product of the two-sphere $S^2$ with itself.
There is a standard area form on $S^2$; the height function
is a momentum map for the circle action that rotates the sphere
around the vertical axis.
Take the product symplectic form on $M$; then
the momentum map for the diagonal circle action sends $(u,v)$
to the sum $u_3 + v_3$.
By~\cite{Ayzenberg:local}, $M/T$ is homeomorphic to the sphere~$S^3$.
Alternatively, this follows from Corollary~\ref{M4}
or from Corollary~\ref{interjections}.

%-------------------------------------------------------------
\item[(6)]
Let $M = \CP^2$, with the Fubini-Study symplectic form,
the circle action given by
$a \cdot [z_0:z_1:z_2] = [a z_0: z_1: z_2]$,
and momentum map 
$[z_0:z_1:z_2] \mapsto \frac{|z_0|^2}{|z_0|^2 + |z_1|^2 + |z_2|^2}$.
By Corollary~\ref{M4}, $M/T$ is homeomorphic to the disc $D^3$.

%-------------------------------------------------------------
\item[(7)]
Let $M = \Sigma_g \times S^2$, where $\Sigma_g$ is a surface
of genus $g$, with the circle acting on the second factor,
a product symplectic form, 
and momentum map $(u,v) \mapsto v_3$.
Then $M/T$ is homeomorphic to $I \times \Sigma_g$,
where $I$ is a closed interval.
This follows from Corollary~\ref{M4} and is also easy to see directly.

Let $\widehat{M}$ be an equivariant symplectic blowup of $M$
at a fixed point.
Then $\widehat{M}$ has one isolated fixed point
and two fixed surfaces of genus $g$.
The quotient $\widehat{M}/T$ is still homeomorphic to $I \times \Sigma_g$.

%-------------------------------------------------------------
\item[(8)]
Let $M = (S^2)^3$ with the product symplectic form,
the  $S^1 \times S^1$ action
$(a,b) \cdot (u,v,w) = (a \cdot u, a \cdot v, b \cdot w)$,
and momentum map $(u,v,w) \mapsto (u_3 + v_3, w_3)$.
The momentum image $\Delta$  is the rectangle $[-2,2] \times [-1,1]$,
and $\Delta_\short = \{-2,2\} \times [-1,1]$.
By Theorem ~\ref{thm:general},
this implies that $M/T$ is  homeomorphic to $S^3 \times I$.
Alternatively, this follows from the facts that $S^2/S^1 \simeq I$
and, as we saw in~(5), that $(S^2)^2/S^1 \simeq S^3$.
Note that in this example $\emptyset \neq \Delta_\short \subsetneq \del \Delta$.

%-------------------------------------------------------------
\item[(9)]
Let $\Sigma_g$ be a surface of genus $g$.
Let $(M,\omega,\mu)$ be any one of the compact symplectic six-manifolds
with Hamiltonian $(S^1)^2$ action and reduced spaces homeomorphic to $\Sigma_g$
that are described in~\cite[Example~1.11]{KT:globex}.
(If $g>0$, there is an infinite number of isomorphism classes
of such manifolds even if we fix the Duistermaat-Heckman measure.)
As is (4), 
the solid dots are the momentum map images of isolated fixed points,
and the segments are the momentum map images 
of 2-spheres with circle stabilizer.
The momentum image $\Delta$
is the closed rectangle whose boundary is marked by dashed lines,
and $\Delta_\short$ is empty.
Theorem~\ref{thm:general} implies that $M/T$ is homeomorphic 
to $I \times I \times \Sigma_g$.

%-------------------------------------------------------------
\item[(10)]
Let $M$ be the projective space $\C \P^5$,
with the three dimensional torus action induced by the
$(S^1)^4$ action on  $\wedge^{2} \C^4 \cong \C^6$,
which itself is induced by the standard action on $\C^4$.
Corollary~12 in \cite[\S{}10]{Buchstaber-Terzic:Gr-CP5} states that 
the quotient $M/T$ is homeomorphic to the join $S^2 \ast \CP^2$.

%-------------------------------------------------------------
\item[(11)] Let $M$ be the Grassmannian of two-planes in $\C^5$,
with the four dimensional torus action induced from the standard
action of $(S^1)^5$ on $\C^5$.
The quotient $M/T$ was studied 
by Buchstaber and Terzic in~\cite{Buchstaber-Terzic:Gr}
and by S\"uss in \cite{Suss2}.

\end{itemize}

%-------------------------------------------------------------------------

%-------------------------------------------------------------------------
\appendix

{
%-------------------------------------------------------------------------
\section{Hamiltonian $T$ actions}
\labell{app:Hamiltonian actions}

A torus $T$ is a Lie group 
that is isomorphic to $(S^1)^r$ for some non-negative integer $r$.
A symplectic manifold is a manifold $M$ equipped with a differential
two-form $\omega$ that is closed and non-degenerate.
A momentum map is a map from the manifold to the dual of the Lie algebra 
of the torus such that,
for every element $\xi$ of the Lie algebra $\t$ of the torus,
the corresponding vector field $\xi_M$ on $M$
(whose value at a point $x \in M$ is 
$\left.\xi_M\right|_{x} = \left.\frac{d}{dt}\right|_{t=0} \exp(t\xi) \cdot x$)
and the corresponding component of the momentum map
$\mu^\xi \colon M \to \R$ 
(whose value at a point $x \in M$ is $\left< \mu(x) , \xi \right>$
where $\left< \cdot , \cdot \right>$ is the pairing between $\t^*$ and $\t$)
are related by Hamilton's equations
\begin{equation} \labell{hamilton}
 d\mu^\xi = - \iota(\xi_M) \omega \quad \text{ for all } \xi \in \t 
\end{equation}
(where $\iota(\xi_M) \omega (v) = \omega(\xi_M,v)$ for any $v \in TM$).
We then call the $T$ action \textbf{Hamiltonian}.

If $M$ is connected, then the affine span of the momentum image $\mu(M)$
is a translate of the annihilator of the Lie algebra of the kernel 
of the action. This is a consequence of Hamilton's equations~\eqref{hamilton}.

The symplectic form is $T$ invariant.
We recall why.
For any $\xi \in \t$, the Lie derivative of $\omega$ along $\xi_M$ 
satisfies $L_{\xi_M} \omega = d\iota(\xi_M) \omega + \iota(\xi_M) d\omega$;
the first summand vanishes because (by Hamilton's equation)
$\iota(\xi_M) \omega$ is exact; the second summand vanishes
because (by assumption) $\omega$ is closed.

The momentum map is constant on orbits.  
We recall why.
For any $\xi, \eta, \zeta \in \t$ we have
$L_{\xi_M} (\omega(\eta_M,\zeta_M))
 = (L_{\xi_M}\omega) (\eta_M,\zeta_M) 
 + \omega ( [\xi_M,\eta_M] , \zeta_M) 
 + \omega ( \eta_M , [\xi_M,\zeta_M])$;
the first summand vanishes because the symplectic form is $T$ invariant;
the second and third summands vanish because $T$ is abelian.
Hence, $\omega(\eta_M,\zeta_M)$ is constant along $T$ orbits.
By Hamilton's equation,
$\omega(\eta_M,\zeta_M) = L_{\eta_M} \mu^{\zeta_M}$;
because for each $T$ orbit the right hand vanishes 
at the point on the (compact) orbit where $\mu^{\zeta_M}$ attains its maximum,
$L_{\eta_M} \mu^{\zeta_M} = 0$.
Because $\eta \in \t$ is arbitrary, $\mu^{\zeta_M}$ is constant
along $T$ orbits; because $\zeta \in \t$ is arbitrary,
$\mu$ is constant along $T$ orbits.

}

%-------------------------------------------------------------------------
\section{Local normal form}
\labell{app:lnf}

A \textbf{Hamiltonian $\bm{T}$ model} is a Hamiltonian $T$-manifold
$(Y,\omega_Y,\mu_Y)$ that is obtained by the following construction.
Let a closed subgroup $H$ of $T$ act on $\C^\ell$ through a homomorphism
$H \to (S^1)^\ell$ followed by 
the standard action of $(S^1)^\ell$ on $\C^\ell$;
the corresponding quadratic momentum map $\mu_H \colon \C^\ell \to \h^*$
is 
$$ z \mapsto \sum_{j=1}^\ell \frac{|z_j|}{2} \eta_j \, ,$$
where $\eta_1,\ldots,\eta_\ell \in \h^*_\Z$
are the weights for the $H$ action on $\C^\ell$.
Take $Y$ to be the manifold $T \times_H (\h^0 \times \C^\ell)$,
where $\h^0$ is the annihilator in $\t^*$ of the Lie algebra of $H$.
Here, we quotient by the anti-diagonal action of $H$,
in which $a \in H$ acts on $T$ by right multiplication by $a^{-1}$,
it acts on $\h^0$ trivially, and it acts on $\C^\ell$ through the given action.
The torus $T$ acts on $Y$ by left multiplication on the $T$ factor.
The \textbf{central orbit} in the model $Y$ is the orbit $[a,0,0]$.

Equip $(T \times \t^*) \times \C^\ell$ with the product
of the standard symplectic form on $T \times \t^*$,
viewed as the cotangent bundle of $T$,
and the standard symplectic form  $\sum_{j=1}^\ell dx_j \wedge dy_j$ 
on~$\C^\ell$.
The pullback of this symplectic form
under the inclusion map 
from $T \times \h^0 \times \C^\ell$ to $(T \times \t^*) \times \C^\ell$
taking $(a,\nu,z)$ to $(a,\nu+\Phi_H(z),z)$
is equal to the pullback of the symplectic form $\omega_Y$ 
under the quotient map 
$T \times \h^0 \times \C^\ell \to T \times_H (\h^0 \times \C^\ell)$;
this determines $\omega_Y$.
Here, we have identified $\t^*$ with $\h^0 \oplus \h^*$.

The momentum map is $\mu_Y([a,\nu,z]) = \alpha + \nu + \mu_H(z)$
for some $\alpha \in \t^*$,
where $\mu_H$ is the quadratic momentum map for the linear $H$ action.

\begin{Theorem}[Local normal form]
\labell{thm:lnf}
Let $M$ be a symplectic manifold with a Hamiltonian $T$ action. 
Then for every orbit in $M$ there exists 
an equivariant symplectomorphism that preserves the momentum maps
from a neighbourhood of the orbit in $M$
to a neighbourhood of the central orbit
in some Hamiltonian $T$ model.
\end{Theorem}

In Theorem~\ref{thm:lnf}, 
the $H$ that appears 
in the Hamiltonian $T$ model that corresponds to the orbit of a point $p$
is the stabilizer of $p$.
Moreover, the weights $\eta_j$ that appear in the model
are unique up to permutation;
we call them the \textbf{isotropy weights} at $p$.

%-------------------------------------------------------------------------

\end{document}